\documentclass[12pt]{amsart}

    \title[NCI-hypergraphs]{A hypergraph characterization of nearly complete intersections}
    \date{\today}
    \author[Bondi]{Chiara Bondi}
    \author[Gibbons]{Courtney R. Gibbons}
    \author[Ke]{Yuye Ke}
    \author[Martin]{Spencer Martin}
    \author[Pothagoni]{Shrunal Pothagoni}
    \author[Stelzer]{Andrew Stelzer}
    \address[Courtney R. Gibbons, Corresponding Author]{ Hamilton College, 198 College Hill Road, Clinton, NY 13323}
    \email[Courtney R. Gibbons]{crgibbon@hamilton.edu}
    \address[Chiara Bondi ]{Hamilton College, 198 College Hill Road, Clinton, NY 13323}
    \address[Yuye Ke]{The Ohio State University, 281 W. Lane Ave, Columbus, OH 43210}
    \address[Spencer Martin]{University of Virginia, Charlottesville, VA 22904}
    \address[Shrunal Pothagoni]{George Mason University, 4400 University Drive, Fairfax, Virginia 22030}
    \address[Andrew Stelzer]{Lawrence University, 711 E. Boldt Way, Appleton, WI 54911}
    
    \thanks{This work was initiated as part of the virtual COURAGE REU in Summer 2020, supported by Clemson University's School of Mathematical and Statistical Sciences.\\
    The authors thank the referee for their helpful and kind comments and uncountably many women in commutative algebra for inspiration, advice, and friendship.}
    \subjclass[2020]{13C05, 13C13, 13C40, 13D02, 13A70}

\usepackage{amsmath, amsthm, amssymb}
\usepackage{tikz}
\usetikzlibrary{topaths,calc}
\usetikzlibrary{shapes,arrows,backgrounds,fit,positioning}
\usepackage{graphicx}
\usepackage{xparse}
\usepackage{mathtools}
\usepackage[disable]{todonotes}
\usepackage{hyperref}
\usepackage{subcaption}

\tikzset{Text/.style={black}}
\tikzset{Vert/.style={circle,fill=orange!50, scale=0.7, draw=black}}
\tikzstyle{edge} = [fill,opacity=.5,fill opacity=.5,line cap=round, line join=round, line width=25pt, draw=black]
\tikzstyle{1edge} = [fill,opacity=.5,fill opacity=.5,line cap=round, line join=round, line width=20pt, draw=black]

\newcommand{\defi}[1]{\textbf{\textit{#1}}}

\newcommand{\lcm}{\text{lcm}}

\NewDocumentCommand\invert{mg}{%
    \ensuremath{\mathcal{I}_{#1}\IfNoValueTF{#2}{}{(#2)}}%
}
\NewDocumentCommand\induced{mg}{%
    \ensuremath{\text{Ind}_{#1}\IfNoValueTF{#2}{}{(#2)}}%
}
\NewDocumentCommand\weakinduced{mg}{%
    \ensuremath{\text{WeakInd}_{#1}\IfNoValueTF{#2}{}{(#2)}}%
}
\NewDocumentCommand\minimalsub{g}{%
    \ensuremath{\text{Min}\IfNoValueTF{#1}{}{(#1)}}%
}

\NewDocumentCommand\pdim{g}{%
    \ensuremath{\text{pd}\IfNoValueTF{#1}{}{(#1)}}%
}
\NewDocumentCommand\join{mg}{%
    \ensuremath{\text{Join}_{#1}\IfNoValueTF{#2}{}{(#2)}}%
}

    \newtheorem{theorem}{Theorem}[section]
    \newtheorem{lemma}[theorem]{Lemma}

    \newtheorem*{thm*}{Theorem}
\theoremstyle{definition}
    \newtheorem{definition}[theorem]{Definition}
    \newtheorem{example}[theorem]{Example}
\theoremstyle{remark}
    \newtheorem{remark}[theorem]{Remark}
\numberwithin{equation}{section}

\begin{document}

    \begin{abstract}
        Recently, nearly complete intersection ideals were defined by Boocher and Seiner to establish lower bounds on Betti numbers for monomial ideals \cite{BoocherSeiner}. Stone and Miller then characterized nearly complete intersections using the theory of edge ideals \cite{StoneMiller}. We extend their work to fully characterize nearly complete intersections of arbitrary generating degrees and use this characterization to compute minimal free resolutions of nearly complete intersections from their degree 2 part.
    \end{abstract}

    \maketitle

    \section{Introduction}\label{sec:Introduction}
        A standard tool for studying a homogeneous ideal $I$ over a polynomial ring $R$ is the minimal free resolution of the module $R/I$. The rank of the $i$-th module in the free resolution is called the $i$-th Betti number of $R/I$ and denoted $\beta_i(R/I)$, and these invariants, along with their graded counterparts (denoted $\beta_{ij}(R/I)$, the $ij$-th Betti number of $R/I$ is the number of degree $j$ minimal generators of the $i$-th syzygy of $R/I$), have been the focus of much study. When $I$ is a complete intersection ideal with $g$ generators, its free resolution is the Koszul complex on its generators, and its $i$-th Betti number is given by $\beta_i(R/I) = \binom{g}{i}$.

        The Buchsbaum-Eisenbud and Horrocks rank conjecture (see \cite[1.4]{BuchBud}) states that an ideal $I$ of height $c$ will always satisfy $\beta_i(R/I) \geq \binom{c}{i}$.  In the case of a complete intersection, it happens that $g = c$ and there is an equality.  When $g > c$ and $I$ is not a complete intersection, much less is known. In \cite{Walker2017}, Walker proved that the sum of the Betti numbers of an arbitrary ideal $I$ is at least $2^c$ if the characteristic of the base field is not 2. In the case of monomial ideals, \cite{BoocherSeiner}, Boocher and Seiner showed that if $I$ is not a complete intersection ideal, then the sum of its Betti numbers $\sum\beta_i(R/I)$ is at least $2^c+2^{c-1}$. Boocher and Seiner's proof centers on reducing arbitrary monomial ideals to a new class they called nearly complete intersections.  They define a \defi{nearly complete intersection} to be a square-free monomial ideal $I$ (not itself a complete intersection) such that for any variable $x$ in the support of $I$, the ideal $I(x=1)$ (formed by substituting in 1 for $x$ in the generators of $I$) is a complete intersection ideal.  Geometrically, one can view a nearly complete intersection as a monomial ideal over a homogeneous ring whose standard affine patches are given by complete intersection ideals.
        
        In \cite{StoneMiller}, Miller and Stone succeeded in classifying nearly complete intersections generated in degree two as the edge ideals of a certain set of simple graphs.
        \begin{thm*}[Miller-Stone 2020] Let $G$ be a connected graph with at least 3 vertices. The edge ideal $I$ of $G$ is \emph{not} a nearly complete intersection (generated in degree 2) if and only if there exist vertices $v_1$, $v_2$, $v_3$, $v_4$, $v_5$ in $G$ such that $v_1$ is a leaf in their induced subgraph $H$ and $H$ has a spanning tree isomorphic to one of the following:
        
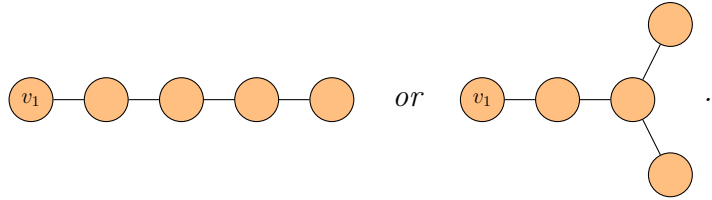
\begin{center}
            \begin{tikzpicture}
                [scale=1,auto=left,every loop/.style={}]
                \node[Vert] (n1) at (0,1) {$v_1$};
                \node[Vert] (n2) at (1,1) {\phantom{$v_2$}};
                \node[Vert] (n3) at (2,1) {\phantom{$v_3$}};
                \node[Vert] (n4) at (3,1) {\phantom{$v_4$}};
                \node[Vert] (n5) at (4,1) {\phantom{$v_5$}};
                
                \draw (n1) -- (n2) -- (n3) -- (n4) -- (n5);
                
                \node[Text] (or) at (5,1) {$\text{or}$};
        
                \node[Vert] (m1) at (6,1) {$v_1$};
                \node[Vert] (m2) at (7,1) {\phantom{$v_2$}};
                \node[Vert] (m3) at (8,1) {\phantom{$v_3$}};
                \node[Vert] (m4) at (8.5,0) {\phantom{$v_4$}};
                \node[Vert] (m5) at (8.5,2) {\phantom{$v_5$}};
                
                \draw (m1) -- (m2) -- (m3) -- (m4);
                \draw (m3) -- (m5);
                
                \node[Text] (punc) at (9,1) {.};
            \end{tikzpicture}
        \end{center}
        \end{thm*}
        
        Our main result, Theorem~\ref{thm:main}, generalizes Miller and Stone's work by classifying all nearly complete intersections as the edge ideals of particularly nice hypergraphs, eliminating the degree restriction. In Section~\ref{sec:Background}, we extend Miller and Stone's notion of vertex inversion to hypergraphs in order to describe nearly complete intersections (hyper)graphically. We prove that this operation commutes with taking induced sub-hypergraphs, laying the foundation for a complete characterization of NCI-hypergraphs in Section~\ref{sec:NCI-hypergraphs}. Section~\ref{sec:NCIres} contains a decomposition of the graded Betti numbers of an arbitrary nearly complete intersection $I$, using an iterated mapping cone construction to reduce the general case to degree 2; we also comment on the effect higher degree generators have on homological invariants of these ideals.
        
        Finally, we include an appendix to discuss an alternative representation of nearly complete intersections as vertex-weighted graphs and share a webtool developed to make the construction and manipulation of these graphs easier.

    \section{Hypergraphs and squarefree monomial ideals}\label{sec:Background}
    
        A \defi{simple hypergraph} $G$ is a pair $(V, E)$ where $V$ is a set and $E \subseteq 2^V$ is a subset of the power set of $V$. We call an element $v\in V$ a \defi{vertex} of $G$ and an element $e\in E$ of cardinality $k$ a \defi{k-edge} of $G$. 
        
        When two elements $e, f\in E$ have the property that $e \subseteq f$, we say that $e$ is a \defi{subedge} of $f$.  If $f$ has no subedges except for itself in $G$, then we call $f$ \defi{minimal}. If every $k$-edge of $G$ is minimal, then we call $G$ a \defi{minimal hypergraph}. Given an arbitrary hypergraph $G$, we can always remove nonminimal edges to get a unique minimal hypergraph, which we denote $\minimalsub{G}$. In this paper, all hypergraphs are assumed to be minimal unless otherwise indicated. We will often refer to 2-edges as ``edges" and $k$-edges for $k\geq3$ as ``hyperedges". When 1-edges occur, we identify them as such.
        
        Given a hypergraph $G = (V, E)$, its \defi{skeleton} is the simple graph $S = (V, E_S)$, where $E_S$ is the set of all 2-edges in $E$. We define the \defi{2-neighbor set} $N(v)$ for a vertex $v$ in $V$ to be its neighbors in the skeleton of $G$, i.e. $N(v) = \{w|\{v, w\}\in E_S\}$. The next definitions provide two distinct yet closely related ways to take sub-hypergraphs of $G$, both of which will be important in this paper.
        
        \begin{definition}\label{defi:induced-hyper}
            An \defi{induced sub-hypergraph} of $G = (V,E)$ is a minimal hypergraph $G' = (V',E')$ where $V' \subseteq V$ and $E' = E \cap 2^{V'}$, the set of all $e\in E$ containing only vertices in $V'$. We denote this sub-hypergraph by $G' = \induced{V'}{G}$
        \end{definition}
        
        \begin{definition}\label{defi:weak-induced-hyper}
            A \defi{weak induced sub-hypergraph} of $G = (V, E)$ is a (not necessarily minimal) hypergraph $G' = (V',E')$ where $V' \subseteq V$ and $E' = \{e \cap V' \vert e \in E\}$. Instead of removing $k$-edges that aren't strictly contained in $V'$, we restrict each one down to the vertices in $V'$. We denote this sub-hypergraph by $G' = \weakinduced{V'}{G}$.
        \end{definition}
        
        Next, we describe how to represent squarefree monomial ideals as minimal hypergraphs. Let $R = k[x_1,...,x_n]$ be a standard graded $k$-algebra, where $k$ is a field. Unless stated otherwise, all ideals are assumed to be squarefree monomial ideals over $R$. The minimal generating set of such an ideal naturally leads to the following definition:
        
        \begin{definition}\label{defi:hypergraph-of-I}
        	Let $I$ be a squarefree monomial ideal with minimal monomial generating set $\{m_1,...,m_\ell\}$ over $R$. The \defi{hypergraph} of $I$ (denoted $G(I)$) is the minimal hypergraph $(V, E)$, where $V$ is the support of $I$ (i.e. all variables of $R$ appearing in some $m_\ell$) and $E = \{m_1,...,m_\ell\}$.
        
        Conversely, given a minimal hypergraph $G = (V, E)$ where $V = \{v_1,...,v_n\}$ and $E = \{m_1,...,m_\ell\}$, the \defi{(hyper)edge ideal} of $G$ is the squarefree monomial ideal $I(G) = (m_1,...,m_\ell)$ over the ring $R = k[v_1,...,v_n]$. Since the minimal monomial generators of a squarefree monomial ideal are unique, this yields a one-to-one correspondence between minimal hypergraphs and squarefree monomial ideals.
        \end{definition}
        
        \begin{example}
            Let $I = (x_1,x_2x_3,x_3x_4x_5)$. Then the hypergraph $G(I)$ consists of a $1$-edge around vertex $x_1$, an edge connecting vertices $x_2$ and $x_3$, and a hyperedge containing vertices $x_3,x_4,$ and $x_5$.
        \end{example}
        
        
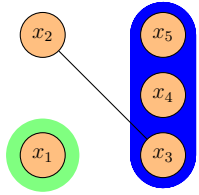
\begin{center}
            \begin{tikzpicture}
                [scale=.4,auto=left,every node/.style={circle,fill=orange!50, scale=0.7, draw=black},every loop/.style={}]
                \node (n1) at (0,0) {$x_1$};
                \node (n2) at (0,4) {$x_2$};
                \node (n3) at (4,0) {$x_3$};
                \node (n4) at (4,2) {$x_4$};
                \node (n5) at (4,4) {$x_5$};
                
                \draw (n2) -- (n3);
                
                \begin{pgfonlayer}{background}
                \begin{scope}[transparency group, opacity=0.5]
                \fill[edge,opacity=1,color=blue, line width=25pt] (n3.center) -- (n4.center) -- (n5.center) -- (n3.center);
                \end{scope}
                \node[1edge, color=green] (bgedge) at (0,0) {};
                \end{pgfonlayer}
            \end{tikzpicture}
        \end{center}
        
        Since nearly complete intersection ideals are squarefree monomial ideals by definition, they can be represented by hypergraphs, and much of this paper is devoted to establishing the structure of these representations. We begin with a trivial definition before translating the algebraic definition of a nearly complete intersection into hypergraph theory.
        
        \begin{definition}
            A hypergraph $G$ is an \defi{NCI-hypergraph} if its (hyper)edge ideal $I(G)$ is a nearly complete intersection. Similarly, $G$ is a \defi{CI-hypergraph} if $I(G)$ is a complete intersection ideal.
        \end{definition}
        
        \begin{lemma}\label{lemma:CI-hypergraph}
            If $G = (V, E)$ is a CI-hypergraph, then each vertex $v$ is contained in at most one $k$-edge $e$ in $E$.
        \end{lemma}
            
            \begin{proof}
                If some $v$ is contained in both $e_1$ and $e_2$, then $I(G)$ must have two minimal generators both divisible by $v$, and therefore $I(G)$ cannot be a complete intersection ideal.
            \end{proof}

        As ideals, nearly complete intersections are defined using a localization operation that amounts to substituting $x=1$ for a variable $x$ in the support of $I$. This operation can be written in the language of hypergraphs we have developed so far.
        
        \begin{definition}
            Let $G = (V,E)$ be a hypergraph and let $v \in V$. The \defi{vertex inversion} of $G$ at $v$ is given by $\invert{v}{G} = \minimalsub{\weakinduced{V \backslash \{v\}}{G}}$
        \end{definition}
        
        \begin{remark}
            Let $I$ be a square-free monomial ideal. An inversion at any vertex $x_i$ in the hypergraph of $I$ yields the hypergraph of $I(x_i=1)$. From this, it immediately follows that a hypergraph $G$ is an NCI-hypergraph iff $\invert{v}{G}$ is a CI-hypergraph for every vertex $v$ in $G$.
        \end{remark}
        
        In later sections we will frequently use the key fact that the operations of vertex inversion and taking induced subhypergraphs commute. Our rigorous proof of this fact is followed by Figure~\ref{fig:comm-ops} which demonstrates the basic idea.
        
        \begin{lemma}\label{lemma:commuting-operations}
            Let $G$ be a hypergraph, let $V'$ be a vertex subset of $G$, and let $v\in V'$. Then $\invert{v}{\induced{V'}{G}} = \induced{V'}{\invert{v}{G}}$.
        \end{lemma}
        
            \begin{proof} By the definition of vertex inversion and the observation that induced and weak induced subgraphs commute, we have
            \begin{align*} \invert{v}{\induced{V'}{G}} &= \minimalsub{\weakinduced{V'\backslash \{v\}}{\induced{V'}{G}}} \\ &=\minimalsub{\induced{V'}{\weakinduced{V\backslash \{v\}}{G}}}.
             \end{align*}
            
            It remains to show that this expression is equal to the induced sub-hypergraph $\induced{V'}{\minimalsub{\weakinduced{V\backslash \{v\}}{G}}}$. To do so, we show that $\minimalsub$ and $\induced{V'}$ commute in general. Consider the poset $\mathcal{P}$ of (hyper)edges in a hypergraph $G$ ordered by inclusion. Then the poset $\mathcal{P}'$ of (hyper)edges in $\induced{V'}{G}$ is simply some subset of $\mathcal{P}$, so any minimal elements of $\mathcal{P}$ contained in $\mathcal{P}'$ remain minimal in $\mathcal{P}'$. Similarly, if $e \in \mathcal{P}'$ is not minimal in $\mathcal{P}$, then there is $e' \subseteq e$ in $\mathcal{P}$. But by the definition of $\induced{V'}{G}$, we have $e' \in \mathcal{P}'$, so $e$ is not minimal in $\mathcal{P}'$. Thus $e \in \mathcal{P}'$ is minimal if and only if it is minimal in $\mathcal{P}$, so $\induced{V'}$ and $\minimalsub$ do indeed commute.
            
            Therefore, \begin{align*} \invert{v}{\induced{V'}{G}} &= \minimalsub{\weakinduced{V'\backslash \{v\}}{\induced{V'}{G}}} \\ &=\minimalsub{\induced{V'}{\weakinduced{V\backslash \{v\}}{G}}} \\
            &= \induced{V'}{\minimalsub{\weakinduced{V\backslash \{v\}}{G}}} \\
            &= \induced{V'}{\invert{v}{G}}.
             \end{align*}
        \end{proof}
            
        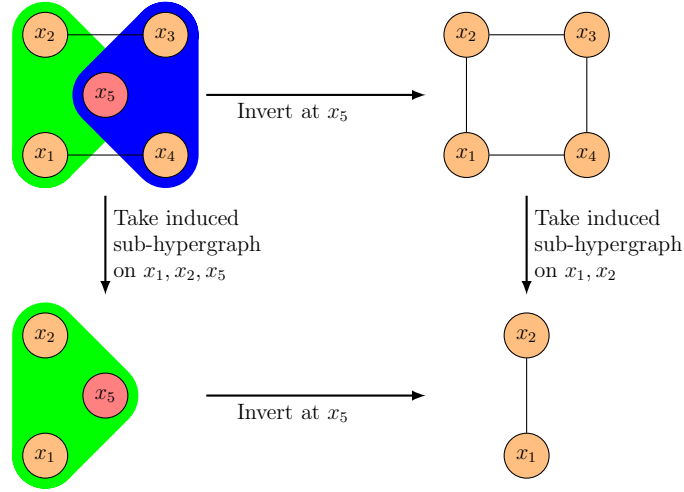
\begin{figure}
            \centering
        	\begin{tikzpicture}
        		[scale=.4,auto=left,every node/.style={circle,fill=orange!50, scale=0.7, draw=black},every loop/.style={}]
          		\node (n1) at (0, 0) {$x_1$};
          		\node (n2) at (0, 4) {$x_2$};
          		\node (n3) at (4, 4) {$x_3$};
          		\node (n4) at (4, 0) {$x_4$};
          		\node[fill=red!50] (n5) at (2, 2) {$x_5$};
          		\node[draw=none, fill=none] (ulr) at (5, 2) {};
          		\node[draw=none, fill=none] (ulb) at (2, -1) {};
          	    
          		\node (m1) at (0 + 14, 0) {$x_1$};
          		\node (m2) at (0 + 14, 4) {$x_2$};
          		\node (m3) at (4 + 14, 4) {$x_3$};
          		\node (m4) at (4 + 14, 0) {$x_4$};
          		
          		\node[draw=none, fill=none] (url) at (-1 + 14, 2) {};
          		\node[draw=none, fill=none] (urb) at (2 + 14, -1) {};
          		
          		\node (nn1) at (0, 0 - 10) {$x_1$};
          		\node (nn2) at (0, 4 - 10) {$x_2$};
          		\node[fill=red!50] (nn5) at (2, 2 - 10) {$x_5$};
          		\node[draw=none, fill=none] (blt) at (2, 5 - 10) {};
          		\node[draw=none, fill=none] (blr) at (5, 2 - 10) {};
          	    
          		\node (mm1) at (0 + 16, 0 - 10) {$x_1$};
          		\node (mm2) at (0 + 16, 4 - 10) {$x_2$};
          		\node[draw=none, fill=none] (brt) at (2 + 14, 5 - 10) {};
          		\node[draw=none, fill=none] (brl) at (-1 + 14, 2 - 10) {};
          		
          		\draw (n1) -- (n4);
          		\draw (n2) -- (n3);
          		
          		\draw (m1) -- (m4);
          		\draw (m1) -- (m2);
          		\draw (m2) -- (m3);
          		\draw (m4) -- (m3);
          		
          		\draw (mm1) -- (mm2);
          		\begin{pgfonlayer}{background}
                \begin{scope}[transparency group, opacity=0.5]
                \fill[edge,opacity=1,color=green] (n1.center) -- (n2.center) -- (n5.center) -- (n1.center);
                \end{scope}
                \begin{scope}[transparency group, opacity=0.5]
                \fill[edge,opacity=1,color=green] (nn1.center) -- (nn2.center) -- (nn5.center) -- (nn1.center);
                \end{scope}
                \begin{scope}[transparency group, opacity=0.5]
                \fill[edge,opacity=1,color=blue] (n3.center) -- (n4.center) -- (n5.center) -- (n3.center);
                \end{scope}
                \end{pgfonlayer}
                
                \draw[-latex,thick] (ulr) -- (url)
                node[midway,below,text width=3cm,fill=none,draw=none,rectangle]{ Invert at $x_5$};
                \draw[-latex,thick] (blr) -- (brl)
                node[midway,below,text width=3cm,fill=none,draw=none,rectangle]{ Invert at $x_5$};
                \draw[-latex,thick] (ulb) -- (blt)
                node[midway,right,text width=3cm,fill=none,draw=none,rectangle]{ Take induced sub-hypergraph on $x_1, x_2, x_5$};
                \draw[-latex,thick] (urb) -- (brt)
                node[midway,right,text width=3cm,fill=none,draw=none,rectangle]{ Take induced sub-hypergraph on $x_1, x_2$};
        		\end{tikzpicture}
        		\caption{Illustration of Lemma \ref{lemma:commuting-operations}.}
        		\label{fig:comm-ops}
        \end{figure}
        
        Before proving results on the hypergraphs of nearly complete intersections, we introduce a special class of hypergraphs that have an almost trivial hyperedge structure.
        
        \begin{definition}
            A hypergraph $G$ is \defi{joinable} if it satisfies the following properties:
            \begin{enumerate}
                \item[(J1)] No two hyperedges of $G$ intersect.
                \item[(J2)] If $e = x_1...x_n$ is a hyperedge of $G$, then $N(x_1)=...=N(x_n)$.
            \end{enumerate}
        \end{definition}
        
        In a joinable hypergraph, hyperedges are only permitted to group disjoint sets of vertices with identical 2-neighbor sets together, creating ``fat vertices" that enable us to view them as simple graphs.
        
        \begin{definition}
            Given a joinable hypergraph $G$, the \defi{join} of $G$ at a hyperedge $h$ is the hypergraph $G' = \textrm{Join}_h(G)$ defined by replacing all vertices in $h$ with a single vertex $v$. (Recall that a 2-edge is not a hyperedge; see Figure~\ref{fig:splay-join-example}.) If all elements of $h$ are connected to a vertex $w$ in $G$, then $v$ is connected to $w$ in $G'$.
        \end{definition}
        
        If no hyperedge is specified, $\textrm{Join}(G)$ refers to the simple graph obtained by iteratively taking the join of $G$ at each of its hyperedges. Note that the order in which the hyperedges are joined into vertices does not matter, so this operation is well-defined. A slightly modified version of the join operation defines a bijection between joinable hypergraphs and vertex-weighted graphs, which makes it possible to describe most of our results in the language of weighted graphs. For more details on this process, see the appendix.

    \section{NCI-hypergraphs}\label{sec:NCI-hypergraphs}
    
        The goal of this section is to prove the following theorem, which is a complete characterization of nearly complete intersections in terms of their corresponding hypergraphs. Recall that a CI-hypergraph is a collection of disjoint (hyper)edges, and that an NCI-hypergraph is a hypergraph $G$ such that $\invert{v}{G}$ is a CI-hypergraph for every vertex $v$ of $G$.
        
        \begin{theorem}\label{thm:main}
            An ideal $I$ is a nearly complete intersection if and only if the corresponding hypergraph $G = G(I)$ satisfies the following properties:
            \begin{itemize}
                \item $G$ is joinable;
                \item The skeleton of $G$ is an NCI-graph.
            \end{itemize}
        \end{theorem}
        
        We prove that NCI-hypergraphs are joinable directly, verifying conditions J1 and J2 separately. The fact that J1 holds for these hypergraphs is essentially a restatement of \cite[Lemma 4.1 and Lemma 4.3]{BoocherSeiner} using our framework. We present our own proof to illustrate how the ideas from that paper can be restated graphically. The fundamental fact needed is Lemma~\ref{lemma:commuting-operations}, which allows us to easily reduce these statements to the local level in proofs by contraposition.
        
        \begin{lemma}\label{lemma:nci-joinable-axiom-1}
            Let $G$ be an NCI-hypergraph. Then no two hyperedges of $G$ intersect.
        \end{lemma}
            
            \begin{proof}
                Let $G$ be a simple hypergraph in which two hyperedges $e_1$ and $e_2$ intersect, and let $H = \induced{e_1\cup e_2}{G}$ be the induced subhypergraph on all elements of $e_1$ and $e_2$. The proof breaks into three cases. First, suppose that the intersection $e_1\cap e_2$ contains at least two vertices $v$ and $w$. If we invert $H$ at $v$, the resulting hypergraph must contain both $e_1\setminus v$ and $e_2\setminus v$, which intersect at $w$ (see Figure~\ref{fig:lem-nci-j1a}). This implies that $\invert{v}{H}$ is not a CI-hypergraph, and by Lemma~\ref{lemma:commuting-operations} this graph is identical to $\induced{e_1\cup e_2}{\invert{v}{G}}$, which is a subgraph of $\invert{v}{G}$. It follows that $G$ cannot be an NCI-hypergraph in this case (see Figure~\ref{fig:lem-nci-j1a}).
                
                Now, suppose instead that $e_1$ and $e_2$ intersect in a single vertex $v$---let $e_1 = (a_1,...,a_n, v)$ and $e_2 = (b_1,...,b_m, v)$. By the previous case we know that $e_1$ and $e_2$ must be the only hyperedges in $H$, and also that $v$ cannot be an element of any edge in $H$. In the case where there are edges in $H$, we can relabel vertices so that we have the edge $(a_1, b_1)$ and then invert $H$ at $v$. We are left with a hypergraph containing the hyperedge $e_1\setminus v$ and edge $(a_1, b_1)$ (among other things), which is not the hypergraph of a complete intersection. If there are no edges in $H$, then inverting $H$ at $a_1$ yields a graph with hyperedges $e_1\setminus a_1$ and $e_2$, which does not represent a complete intersection either.  Applying Lemma~\ref{lemma:commuting-operations} again proves that $G$ cannot be an NCI-hypergraph (see Figure~\ref{fig:lem-nci-j1b}). \end{proof}
                
                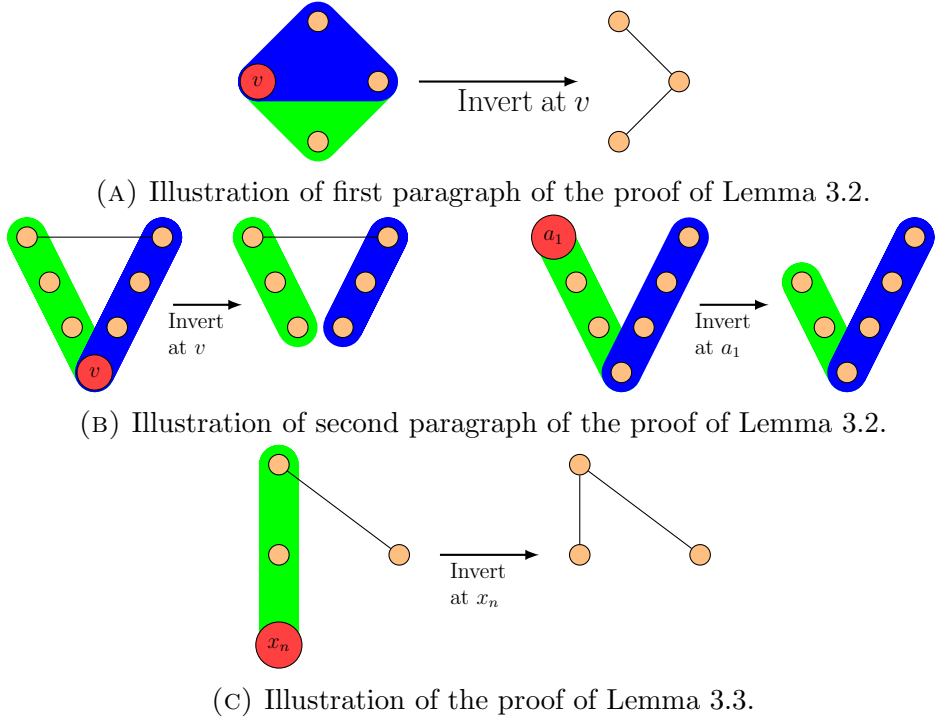
\begin{figure}
                \centering
                \begin{subfigure}{\textwidth}
                \centering
                    \begin{tikzpicture}
                        [scale=.4,auto=left,every node/.style={circle,fill=orange!50, scale=0.7, draw=black},every loop/.style={}]
                        \node (n1) at (2, 0) {};
              		    \node[fill=red!75] (n2) at (0, 2) {$v$};
              		    \node (n3) at (2, 4) {};
              		    \node (n4) at (4, 2) {};
              		    \node[fill=none, draw=none] (1AL) at (5, 2) {};
              		    
              		    \begin{pgfonlayer}{background}
              		    \begin{scope}[transparency group, opacity=0.5]
                        \fill[edge,opacity=1,color=green, line width=15pt] (n2.center) -- (n1.center) -- (n4.center) -- (n2.center);
                        \end{scope}
              		    \begin{scope}[transparency group, opacity=0.5]
                         \fill[edge,opacity=1,color=blue, line width=15pt] (n2.center) -- (n3.center) -- (n4.center) -- (n2.center);
                        \end{scope}
                        \end{pgfonlayer}
                        
              		    \node (n1) at (2 + 10, 0) {};
              		    \node (n3) at (2 + 10, 4) {};
              		    \node (n4) at (4 + 10, 2) {};
              		    \node[fill=none, draw=none] (1AR) at (11, 2) {};

              		    \draw (n3) -- (n4);
              		    \draw (n1) -- (n4);
                    
                        \draw[-latex,thick] (1AL) -- (1AR)
                        node[below,text width=5cm,fill=none,draw=none,rectangle]{ \Large Invert at $v$};
                        \end{tikzpicture}
                    \subcaption{Illustration of first paragraph of the proof of Lemma~\ref{lemma:nci-joinable-axiom-1}.}
                    \label{fig:lem-nci-j1a}
                \end{subfigure}
                        
                \begin{subfigure}{\textwidth}
                    \begin{tikzpicture}
                		[scale=.3,auto=left,every node/.style={circle,fill=orange!50, scale=0.7, draw=black},every loop/.style={}]
                		\node (n1) at (0, 8) {};
                  		\node (n2) at (1, 6) {};
                  		\node (n3) at (2, 4) {};
                  		\node[fill=red!75] (n4) at (3, 2) {$v$};
                  		\node (n5) at (4, 4) {};
                  		\node (n6) at (5, 6) {};
                  		\node (n7) at (6, 8) {};
                  		\node[fill=none, draw=none] (1AL) at (6, 5) {};
                  	
                  		\draw (n1) -- (n7);
                  		
                  		\begin{pgfonlayer}{background}
                        \begin{scope}[transparency group, opacity=0.5]
                        \fill[edge,opacity=1,color=green, line width=15pt] (n1.center) -- (n2.center) -- (n3.center) -- (n4.center) -- (n1.center);
                        \end{scope}
                        \begin{scope}[transparency group, opacity=0.5]
                        \fill[edge,opacity=1,color=blue, line width=15pt] (n7.center) -- (n6.center) -- (n5.center) -- (n4.center) -- (n7.center);
                        \end{scope}
                        \end{pgfonlayer}
                        
                        \node (m1) at (0 + 10, 8) {};
                  		\node (m2) at (1 + 10, 6) {};
                  		\node (m3) at (2 + 10, 4) {};
                  	
                  		\node (m5) at (4 + 10, 4) {};
                  		\node (m6) at (5 + 10, 6) {};
                  		\node (m7) at (6 + 10, 8) {};
                  		\node[fill=none, draw=none] (1AR) at (10, 5) {};
                  	
                  		\draw (m1) -- (m7);
                  		
                  		\begin{pgfonlayer}{background}
                        \begin{scope}[transparency group, opacity=0.5]
                        \fill[edge,opacity=1,color=green, line width=15pt] (m1.center) -- (m2.center) -- (m3.center) -- (m1.center);
                        \end{scope}
                        \begin{scope}[transparency group, opacity=0.5]
                        \fill[edge,opacity=1,color=blue, line width=15pt] (m7.center) -- (m6.center) -- (m5.center) -- (m7.center);
                        \end{scope}
                        \end{pgfonlayer}
                        
                        \draw[-latex,thick] (1AL) -- (1AR)
                        node[midway,below,text width=1.5cm,fill=none,draw=none,rectangle]{ Invert at $v$};
                	\end{tikzpicture}{\hspace{1.5cm}}
                  	\begin{tikzpicture}
                		[scale=.3,auto=left,every node/.style={circle,fill=orange!50, scale=0.7, draw=black},every loop/.style={}]
                		\node[fill=red!75] (n1) at (0, 8) {$a_1$};
                  		\node (n2) at (1, 6) {};
                  		\node (n3) at (2, 4) {};
                  		\node (n4) at (3, 2) {};
                  		\node (n5) at (4, 4) {};
                  		\node (n6) at (5, 6) {};
                  		\node (n7) at (6, 8) {};
                  		\node[fill=none, draw=none] (1AL) at (6, 5) {};
                  		
                  		\begin{pgfonlayer}{background}
                        \begin{scope}[transparency group, opacity=0.5]
                        \fill[edge,opacity=1,color=green, line width=15pt] (n1.center) -- (n2.center) -- (n3.center) -- (n4.center) -- (n1.center);
                        \end{scope}
                        \begin{scope}[transparency group, opacity=0.5]
                        \fill[edge,opacity=1,color=blue, line width=15pt] (n4.center) -- (n5.center) -- (n6.center) -- (n7.center) -- (n7.center);
                        \end{scope}
                        \end{pgfonlayer}
                        
                  		\node (m2) at (1 + 10, 6) {};
                  		\node (m3) at (2 + 10, 4) {};
                  		\node (m4) at (3 + 10, 2) {};
                  		\node (m5) at (4 + 10, 4) {};
                  		\node (m6) at (5 + 10, 6) {};
                  		\node (m7) at (6 + 10, 8) {};
                  		\node[fill=none, draw=none] (1AR) at (10, 5) {};
                  		
                  		\begin{pgfonlayer}{background}
                        \begin{scope}[transparency group, opacity=0.5]
                        \fill[edge,opacity=1,color=green, line width=15pt] (m2.center) -- (m3.center) -- (m4.center) -- (m2.center);
                        \end{scope}
                        \begin{scope}[transparency group, opacity=0.5]
                        \fill[edge,opacity=1,color=blue, line width=15pt] (m7.center) -- (m6.center) -- (m5.center) -- (m4.center) -- (m7.center);
                        \end{scope}
                        \end{pgfonlayer}
                        
                        \draw[-latex,thick] (1AL) -- (1AR)
                        node[midway,below,text width=1.5cm,fill=none,draw=none,rectangle]{ Invert at $a_1$};
            	    \end{tikzpicture}
        	    \subcaption{Illustration of second paragraph of the proof of Lemma~\ref{lemma:nci-joinable-axiom-1}.}
        	    \label{fig:lem-nci-j1b}
        	\end{subfigure}
        	
        	\begin{subfigure}{\textwidth}
        	\centering
        	    \begin{tikzpicture}
              		[scale=.4,auto=left,every node/.style={circle,fill=orange!50, scale=0.7, draw=black},every loop/.style={}]
              		\node (n1) at (0, 6) {};
              	    \node (n2) at (0, 3) {};
              		\node[fill=red!75] (n3) at (0, 0) {$x_n$};
              		\node (n4) at (4, 3) {};
              	    \node[fill=none, draw=none] (l) at (5, 3) {};
              	
              		\draw (n1) -- (n4);
              		
              		\begin{pgfonlayer}{background}
                    \begin{scope}[transparency group, opacity=0.5]
                    \fill[edge,opacity=1,color=green, line width=15pt] (n1.center) -- (n2.center) -- (n3.center) -- (n1.center);
                    \end{scope}
                    \end{pgfonlayer}

              		\node (m1) at (0 + 10, 6) {};
              	    \node[fill=none, draw=none] (r) at (0 + 9, 3) {};
              		\node (m3) at (0 + 10, 3) {};
              		\node (m4) at (4 + 10, 3) {};
              	
              		\draw (m1) -- (m4);
              		\draw (m1) -- (m3);
              		
              		\draw[-latex,thick] (l) -- (r)
                    node[midway,below,text width=1.5cm,fill=none,draw=none,rectangle]{ Invert at $x_n$};
          	    \end{tikzpicture}
            \subcaption{Illustration of the proof of Lemma~\ref{lemma:nci-joinable-axiom-2}.}
            \label{fig:lem-nci-j2}
        	\end{subfigure}
        	\caption{Proof illustrations.}
            \end{figure}

        Condition J2 (elements of hyperedges have the same 2-neighbors) is shown to be necessary in identical fashion:
        
        \begin{lemma}\label{lemma:nci-joinable-axiom-2}
            If $H$ is an NCI-hypergraph and $e = (x_1,...,x_n)$ is a hyperedge of $H$, then $N(x_1)=...=N(x_n)$.
        \end{lemma}
            
            \begin{proof}
                Let $G$ be a simple hypergraph, $e = (x_1,...,x_n)$ a hyperedge of $G$, and $v$ a vertex of $G$ connected to $x_1$ by an edge but not to $x_n$. Let $H = \induced{e\cup\{v\}}{G}$ be the subhypergraph induced by $\{x_1,...,x_n, v\}$ and consider its inversion $\invert{x_n}{H}$ at $x_n$ (see Figure~\ref{fig:lem-nci-j2}).
        	    
                The resulting hypergraph must contain both $(x_1, v)$ and $(x_1,...,x_{n-1})$ and thus does not represent a complete intersection. Lemma~\ref{lemma:commuting-operations} then implies that the inversion of $G$ at $x_n$ does not represent a complete intersection either, so $G$ cannot be an NCI-hypergraph. 
            \end{proof}
        
        To prove Theorem~\ref{thm:main}, it now suffices to show that if $G = (V, E)$ is a joinable NCI-hypergraph, then adding or removing a hyperedge from $G$ gives another NCI-hypergraph. The proof of each case here is brief.
        
        \begin{proof}[Proof of Theorem~\ref{thm:main}]
        
        First, let $e$ be a hyperedge of $G$. Since $G$ is joinable, only 2-edges intersect with $e$. The cardinality of $e$ is still at least 2 in the inversion $\invert{v}{G}$ for any $v\in V$, so $e$ does not delete any elements of $E$ when Min is applied. It follows that $\invert{v}{G-\{e\}}$ is a subgraph of $\invert{v}{G}$, which implies that $\invert{v}{G-\{e\}}$ represents a complete intersection for all $v\in V$. Thus $G-\{e\}$ is an NCI-hypergraph as desired.
        
        In the other direction, suppose that $W = (v_1,...,v_n)\subset V$ is an independent set of vertices ($n\geq 3$) such that $G' = (V, E\cup\{W\})$ is still a joinable hypergraph. Since $N(v_1)=...=N(v_n)$ and $G$ is an NCI-hypergraph by assumption, all of the $v_i$ must become isolated after inversion at any $v\in V$. This implies that in $\invert{v}{G'}$ the $v_i$ are either isolated as in $G$ or grouped into the single hyperedge $W$. In either case every vertex is contained in at most one $k$-edge in this inversion, so $G'$ is still an NCI-hypergraph.
        
        The proof of Theorem~\ref{thm:main} follows quickly now. If $G$ is an NCI-hypergraph, then it must be joinable by Lemmas \ref{lemma:nci-joinable-axiom-1} and \ref{lemma:nci-joinable-axiom-2}, 
        and the above discussion implies that its skeleton must also be an NCI-hypergraph, so $G$ satisfies the hypotheses of Theorem~\ref{thm:main}. Conversely, if $G$ satisfies the hypotheses of the theorem then its skeleton is an NCI-hypergraph by assumption, and we can add on its hyperedges one by one without changing this fact.
        \end{proof}
        
        Given an arbitrary ideal $I$, it is fairly easy to check whether or not $G(I)$ is a joinable hypergraph, so the utility of Theorem~\ref{thm:main} hinges entirely on our ability to identify nearly complete intersections generated in degree 2. In \cite{StoneMiller}, Stone and Miller provide a characterization of the graphs of these ideals using two small forbidden subgraphs, completing our description of NCI-hypergraphs. We close this section with a lemma describing the effect of the join operation on NCI-hypergraphs, which will be useful later.
        
        \begin{lemma}\label{lemma:join-preserves-nci}
            Let $G = (V, E)$ be an NCI-hypergraph and let $h \in E$ be a hyperedge. Then $\join{h}{G}$ is an NCI-hypergraph or a CI-hypergraph.
            
            \begin{proof}
            Let $v \in h$, and let $V' = (V \backslash h) \cup v$. Then $\join{h}{G} \cong \induced{V'}{G}$. We wish to show that $\invert{v}{\induced{V'}{G}}$ is a CI-hypergraph. By the commutativity of vertex inversion with taking induced sub-hypergraphs, we have that $\invert{w}{\induced{V'}{G}} = \induced{V'}{\invert{w}{G}}$ for any $w \in V'$. Since $G$ is an NCI-hypergraph, $\invert{w}{G}$ is always a CI-hypergraph, meaning $\induced{V'}{\invert{w}{G}}$ is also a CI-hypergraph. Hence, inverting at any vertex in $\join{h}{G}$ yields a CI-hypergraph. Thus, $\join{h}{G}$ must be an NCI-hypergraph or a CI-hypergraph.
            \end{proof}
        \end{lemma}
    
    \section{Resolutions of NCIs}\label{sec:NCIres}
    
        Using the established structure of NCI-hypergraphs, we can now extract information about their Betti numbers and minimal free resolutions. We explicitly describe the effect of the generators of degree $3$ and higher on the Betti table of a nearly complete intersection. In this section, we gently abuse terminology by referring to the skeleton of a nearly complete intersection ideal $I$; by this we mean the degree 2 squarefree monomial ideal corresponding to the skeleton of the induced hypergraph $G(I)$.
        
        Given the minimal free resolution of the skeleton of a nearly complete intersection, the minimal free resolution of the entire nearly complete intersection is obtained through the use of an iterated mapping cone construction. To prove the minimality of this constructed resolution, it suffices to show that each step is a Betti splitting of our ideal as defined in \cite{Monica}.
        We take this opportunity to remind the reader that the Betti numbers of $I$ are slightly offset from the Betti numbers of $R/I$, with $\beta_{ij}(I) = \beta_{i+1\, j}(R/I)$.
        
        \begin{definition}[from \cite{Monica}]
         $I = J+K$ is a \defi{Betti splitting} if for all $i, j\geq 0$ we have
            $$\beta_{ij}(I) = \beta_{ij}(J)+\beta_{ij}(K)+\beta_{i-1, j}(J\cap K)$$    
        \end{definition}
        
        Note that if $I$ is generated by $\{m_1,...,m_n\}$, any bipartition of the $m_i$ provides a candidate decomposition $J+K$. From the definition, it is not immediately obvious how to check if a given partition yields a Betti splitting without explicitly computing the Betti tables. Fortunately, Eliahou and Kervaire provided a more applicable condition for proving a partition is a Betti splitting in the ungraded case \cite{KE-Splittings}. This result was extended by Fatabbi to show that such Betti splittings respect graded Betti numbers \cite{FatPoints}. Let $\mathcal{G}(I)$ represent the generators of $I$.
        
        \begin{definition}[from \cite{KE-Splittings}]
            We say that $J+K = I$ is a \defi{splitting} of $I$ if $\mathcal{G}(J)$ and $\mathcal{G}(K)$ form a partition of $\mathcal{G}(I)$ and there exists a splitting function $\mathcal{G}(J\cap K)\rightarrow \mathcal{G}(J)\times \mathcal{G}(K)$ sending $w$ to $(\phi(w), \psi(w))$ and satisfying the following two properties:\\
            \begin{enumerate}
                \item[(S1)] For all $w\in \mathcal{G}(J\cap K)$, $w = \lcm(\phi(w), \psi(w))$.
                \item[(S2)] For every subset $G'\subset \mathcal{G}(J\cap K)$, both $\lcm(\phi(G'))$ and $\lcm(\psi(G'))$ strictly divide $\lcm(G')$.
            \end{enumerate}
        \end{definition}
        
        Defining an appropriate splitting function is more feasible than attempting to explicitly compute free resolutions of arbitrary nearly complete intersections. We prove that any nearly complete intersection $I$ with at least one generator $h$ of degree at least $3$ has a Betti splitting $I = (h) + K$. Applying this decomposition iteratively allows us to break the Betti table of $I$ down into a sum of three sets of tables: the table of the skeleton of $I$, trivial tables coming from each principal ideal $(h)$, and tables of intersection ideals $(h)\cap K$. We also prove that $(h)\cap K$ is a multiple of a complete intersection.
        
        \begin{theorem}\label{thm:hyperedge-betti-splitting}
            Let $I = (h) + K$ be a nearly complete intersection where $h$ is a monomial of degree at least 3. Then $(h)+K$ is a Betti splitting of $I$.
        \end{theorem}
            
            \begin{proof}
                If $K$ is generated by $(m_1,...,m_n)$, then a (not necessarily minimal) set of generators of $(h)\cap K$ is given by $\{\lcm(hm_i)\}_{i=1}^n$. From our knowledge of the hypergraph structure of $G(I)$, we know that each $m_i$ has one of two forms: either it can be written as $xy$, where $x$ lies in $h$ and $y$ is connected to every vertex of $h$, or $m_i = e$ represents a (hyper)edge of $G(I)$ that does not intersect $h$. In the first case, $\lcm(h,m_i) = \lcm(h,xy) = hy$, while in the second case $\lcm(h,m_i) = \lcm(h,e) = he$. Some of these generators may generate each other, in which case we remove those terms to get a minimal generating set where all generators are of one of these forms.
                
                We can now explicitly construct a splitting function $\mathcal{G}((h)\cap K)\rightarrow \mathcal{G}((h))\times \mathcal{G}(K)$. Define the function $w\mapsto(\phi(w), \psi(w))$ as follows: $\phi(w) = h$, $\psi(he) = e$, and $\psi(hy) = xy$, where $x$ is some fixed vertex chosen in $h$. We now verify the two properties of this splitting function in turn.
        
                (S1): We need to show that for all $w$ generating the intersection, $w = \lcm(\phi(w), \psi(w))$. If $w = he$, then $\phi(w) = h$ and $\psi(w) = e$, so their least common multiple is indeed $he = w$. If $w = hy$, then $\phi(w) = h$ and $\psi(w) = xy$, and since $x\in h$ and $y\notin h$ by definition their least common multiple is $hy = w$.
        
                (S2): We need to show that for every subset $G'$ of generators in the intersection, both $\lcm(\phi(G'))$ and $\lcm(\psi(G'))$ strictly divide $\lcm(G')$. We know that $\lcm(\phi(G'))$ is always just $h$, which certainly divides $\lcm(G')$. The set $\psi(G')$ consists of $m$ distinct edges $e_1,...,e_m$ and $n$ distinct terms of the form $xy_1,...,xy_n$. The lcm of all these terms must be $\lcm(x, e_1, \dots, e_m, y_1, \dots, y_n)$. This strictly divides $\lcm(G')$, which is the same $\lcm$ with $x$ replaced by all of $h$.
            \end{proof}

        Take an arbitrary nearly complete intersection $I$, call its skeleton $S$, and denote the hyperedges of its hypergraph by $h_1, h_2,...,h_n$. For $1\leq k\leq n$, let $K_j$ denote the ideal generated by all generators of $I$ except for $h_1,...,h_j$. Applying the above theorem to $I$ repeatedly yields the following:
        
        \begin{equation}\label{eq:beta-I}\beta_{ij}(I) = \beta_{ij}(S) + \sum_{a=1}^n\beta_{ij}((h_a)) + \sum_{b=1}^n\beta_{i-1 \, j}((h_b)\cap K_b).\end{equation}
        
        The minimal free resolution of $R/(h_a)$ is 
        \begin{equation}\label{eq:free-res-ha}
        0 \leftarrow R \leftarrow R(- \deg(h_a)) \leftarrow 0\end{equation} so its Betti table is trivial:
        
        \[ \beta_{ij}((h_a)) = \begin{cases} 
                  1 & (i, j) = (0, 0) \text{ or } (1, \deg(h_a)), \\
                  0 & \text{otherwise}.
               \end{cases}
        \]
        
        The following theorem can be used to compute the Betti numbers of $(h_i)\cap K_i$.
        
        \begin{theorem}\label{thm:minimal-resolution-intersection-is-ci}
        Let $I$ be a nearly complete intersection and $h$ be a generator of degree at least 3, so $I = (h)+K$ is a Betti splitting of $I$. Then $(h)\cap K$ is of the form $hJ$, where $J$ is a complete intersection.
        \end{theorem}
        
             \begin{proof}
                If $(m_1, \dots, m_k)$ is the minimal monomial generating set for $K$, then $(\lcm(m_1, h), \dots, \lcm(m_k, h))$ is a monomial generating set for $(h) \cap K$. Moreover, $\lcm(m_i, h) = hm'_i$ for some square-free monomial $m'_i$ coprime to $h$. Let $J = (m'_1, \dots, m'_k)$. We wish to show that $J$ is a complete intersection by reinterpreting its construction as operations on a hypergraph.
                
                Let $H = (V, E)$ be a hypergraph and define the hypergraph operation $\mathcal{L}_{v}(H)$ to be $\mathcal{L}_{v}(H) = (V, E')$ where $E' = \{e \cup \{v\} \vert e \in E\}$. Then we may translate the construction of $J$ into hypergraph operations as follows:
                
                Considering the monomial $h$ as a formal variable corresponds to taking the join of the hypergraph of $K$ at $h$. Computing the ideal generated by $\lcm(m_i, h)$ then corresponds to computing $\mathcal{L}_h(H)$ on the resulting hypergraph. Finally, factoring out the $h$ term is equivalent to inverting at $h$ and then multiplying the resulting edge ideal by $h$.
                
                Hence, if $H$ is the hypergraph of $K$, we get that the hypergraph of $J$ is given by $(\invert{h} \circ \mathcal{L}_h \circ \text{Join}_h)(H)$. But $\mathcal{I}_h = \minimalsub \circ \weakinduced{V \backslash \{v\}}$ and $\weakinduced{V \backslash \{h\}} \circ \mathcal{L}_h = \weakinduced{V \backslash \{h\}}$. Hence, the hypergraph of $J$ is given by $(\invert{h} \circ \text{Join}_h)(H)$. By Lemma~\ref{lemma:join-preserves-nci}, the join of an NCI-hypergraph at a hyperedge yields an NCI-hypergraph or a CI-hypergraph, and in either case inverting at any vertex in the support of the resulting hypergraph yields a CI-hypergraph. Thus $(\invert{h} \circ \text{Join}_h)(H)$ is a CI-hypergraph, so $J$ is a complete intersection. \end{proof}
        
        Since $(h) \cap K = hJ$ for some complete intersection $J$ with monomial generators $f_1, f_2, \ldots, f_d$, we can explicitly calculate the free resolution of $R/hJ$.  The first syzygy is generated by $hf_1, hf_2, \ldots, hf_d$ since the support of $h$ and the support of $J$ are disjoint. The second syzygy is straightforward to calculate explicitly, and it is exactly the second syzygy $R/J$.  It follows that the remaining syzygies are the corresponding Koszul relations of $R/J$, and thus the projective dimension of $(h) \cap K$ is the number of minimal generators of $hJ$.
        
        The skeleton $S$ is much more complicated, and we note only that our decomposition technique cannot be used to break it down further, as the principal ideal generated by a 2-edge of $I$ does not work with the splitting function defined in Theorem~\ref{thm:hyperedge-betti-splitting}. In this paper we will not prove results about the structure of $\beta_{ij}(S)$.
        
        \begin{example}\label{ex:inductive-splitting}
        Consider the nearly complete intersection $I$ given by the hypergraph below (over the ring $R = k[x_1,\ldots,x_8]$). The hyperedges correspond to the monomials $H = \{x_1x_2x_3, x_4x_5x_6\}$, the skeleton corresponds to the monomials in $S = \{x_ix_7, x_ix_8, x_7x_8 \, | \, 1 \leq i \leq 6\}$, and the ideal $I$ is generated by the union $S \cup H$.
        
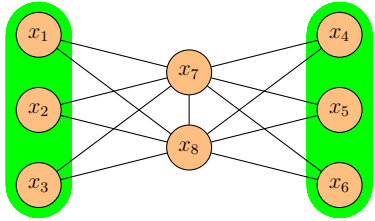
            \begin{center}
                \begin{tikzpicture}
              		[scale=1,auto=left,every node/.style={circle,fill=orange!50, scale=0.7, draw=black},every loop/.style={}]
              		\node (n1) at (0, 3) {$x_1$};
              	    \node (n2) at (0, 2) {$x_2$};
              		\node (n3) at (0, 1) {$x_3$};
              		\node (n7) at (2, 2.5) {$x_7$};
              		\node (n8) at (2, 1.5) {$x_8$};
              		\node (n4) at (4, 3) {$x_4$};
              		\node (n5) at (4, 2) {$x_5$};
              		\node (n6) at (4, 1) {$x_6$};
              	
              		\draw (n7) -- (n1);
              		\draw (n7) -- (n2);
              		\draw (n7) -- (n3);
              		\draw (n7) -- (n4);
              		\draw (n7) -- (n5);
              		\draw (n7) -- (n6);
              	    \draw (n8) -- (n1);
              		\draw (n8) -- (n2);
              		\draw (n8) -- (n3);
              		\draw (n8) -- (n4);
              		\draw (n8) -- (n5);
              		\draw (n8) -- (n6);
              		\draw (n7) -- (n8);
              		
              		\begin{pgfonlayer}{background}
                    \begin{scope}[transparency group, opacity=0.5]
                    \fill[edge,opacity=1,color=green, line width=25pt] (n1.center) -- (n2.center) -- (n3.center) -- (n1.center);
                    \fill[edge,opacity=1,color=green, line width=25pt] (n4.center) -- (n5.center) -- (n6.center) -- (n4.center);
                    \end{scope}
                    \end{pgfonlayer}
          	    \end{tikzpicture}
          	\end{center}
          	
        With this setup, $h_1 = x_1x_2x_3$, $h_2 = x_4x_5x_6$, $K_1 = (S) + (h_2)$, and $K_2 = (S)$. For the first step of the decomposition in Equation \ref{eq:beta-I}, we write $\beta_{ij}(I) = \beta_{ij}(K_1) + \beta_{ij}(h_1) + \beta_{i-1\,j}(K_1\cap(h_1))$ and compute the intersection ideal $K_1\cap(h_1) = h_1(x_7, x_8, h_2)$. In the second step, we decompose $\beta_{ij}(K_1) = \beta_{ij}(S) + \beta_{ij}(h_2)+\beta_{i-1\,j}(K_2\cap(h_2))$ and find that $K_2\cap(h_2) = K_2\cap(h_2) = h_2(x_7, x_8)$. Both intersection ideals are monomial multiples of complete intersections, as guaranteed by Theorem \ref{thm:minimal-resolution-intersection-is-ci}, so four of the five Betti numbers in this decomposition are easy to compute.  Assuming we have computed $\beta(K_1)$ as in step 2, the component Betti tables from step 1 sum to the Betti table for $I$ as follows:
        \begingroup
        \allowdisplaybreaks
        \begin{footnotesize}
        \begin{align*}
        \beta(I) &= 
        \begin{pmatrix}
        13 & 42 & 70 & 70 & 42 & 14 & 2\\
        1  & 2  & 1  &  - & -  & -  & -\\
         - & -  & -  & -  & -  & -  & -\\
         - & -  & -  & -  & -  & -  & -
        \end{pmatrix} & \text{ from $\beta(K_1)$ }\\
        & \quad + \begin{pmatrix}
            - & - & - & - & - & - & -\\
            1 & - & - & - & - & - & -\\
            - & -  & -  & -  & -  & -  & -\\
            - & -  & -  & -  & -  & -  & -
            \end{pmatrix} & \text{ from $\beta(h_1)$}\\
        & \quad + \begin{pmatrix}
            - & - & - & - & - & - & - \\
            - & 2 & 1 & - & - & - & - \\
            - & - & - & - & - & - & - \\
            - & 1 & 2 & 1 & - & - & - 
            \end{pmatrix} & \text{from the shift of $\beta(K_1\cap(h_1))$} \\
        & =
        \begin{pmatrix}
        13 & 42 & 70 & 70 & 42 & 14 & 2\\
        2  & 4  & 2  & -  & -  & -  & -\\
        -  & -  & -  & -  & -  & -  & -\\
        -  & 1  & 2  & 1  & -  & -  & -\\
        \end{pmatrix}.
        \end{align*}
        \end{footnotesize}
        \endgroup
        \end{example}
        
        Note in particular that the projective dimension of every component other than $S$ in Example~\ref{ex:inductive-splitting} is less than 4, while the dimension of the skeleton is 7. This suggests that the projective dimension of a nearly complete intersection is primarily determined by the skeleton of its corresponding hypergraph. The following theorem captures some of this intuition:
        
        \begin{theorem}
        Let $I = K + (h)$ be a nearly complete intersection with $\deg(h)\geq 3$, and let $\pdim$ denote the projective dimension of an ideal. Then $\pdim{I} \leq \pdim{K} + 1$.
        \end{theorem}
        
        \begin{proof}
        Suppose $I = K+(h)$ is as above. Then $K + (h)$ is a Betti splitting of $I$, which implies that
        $$\pdim{I} = \max(\pdim{K}, \pdim{(h)}, \pdim{K \cap  (h)} + 1)$$
        We know that $\pdim{(h)} = 1$, so it suffices to show that $\pdim{K \cap  (h)}\leq \pdim(K)$. We accomplish this by demonstrating that $\pdim{K \cap (h)} = \pdim{K(x=1)}$ for any variable $x$ dividing $h$ and then noting that $\pdim{K(x=1)}$ cannot exceed $\pdim{K}$.
        
        By Theorem~\ref{thm:minimal-resolution-intersection-is-ci}, $K \cap (h) = hJ$ where $J$ is a complete intersection. Similarly, $K(x=1)$ is a complete intersection. Since both ideals are complete intersections, their minimal free resolutions are given by the Koszul complexes on their minimal monomial generators. This implies in particular that $\pdim(J)$ is equal to the number of (hyper)edges in $G(J)$ and $\pdim(K(x=1))$ is equal to the number of (hyper)edges in $G(K(x=1))$.
        
        By the proof of Theorem~\ref{thm:minimal-resolution-intersection-is-ci}, we have that $G(J) = \invert{h}{\join{h}{G(K)}}$, and the proof of Lemma~\ref{lemma:join-preserves-nci} implies further that $G(J) \cong \invert{x}{\induced{V'}{G(K)}}$ where $V' = (V \backslash h) \cup \{x\}$ and $V$ is the collection of vertices in $G(K)$. Since $\induced{V'}$ and $\invert{x}$ commute by Lemma~\ref{lemma:commuting-operations}, we can write $G(J) \cong \induced{V'}{\invert{x}{G(K)}}$. Inverting at $x$ in $G(K)$ kills all edges connected to every vertex of the hyperedge $h$ because $N(x) = N(y)$ for every vertex $y$ in $h$, so $\induced{V'}{\invert{x}{G(K)}}$ has the same number of (hyper)edges as $\invert{x}{G(K)}$. Thus $G(J)$ and $G(K(x=1))$ have the same number of (hyper)edges, implying that their projective dimensions are equal.
        
        Since considering the ideal $K(x = 1)$ is analogous to localizing $K$ with respect to the powers of $x$, it follows that $\pdim{K(x=1)} \leq \pdim{K}$ \todo{what's the best lingo to use here?} and the proof is complete.
        \end{proof}
        
        Although we only have that $\pdim{K(x=1)}\leq\pdim{K}$ for any variable $x$ belonging to $h$, we conjecture that the inequality is always strict, which would imply that $\pdim{K + (h)} = \pdim{K}$. An important consequence of this conjecture is that the projective dimension of a nearly complete intersection should be completely determined by the skeleton of its hypergraph.
        
        Not all homological properties of nearly complete intersections appear to be as well-behaved as projective dimension, however. The regularity of these ideals, for example, may depend on either the skeleton or the higher-degree terms. The involvement of higher-degree terms is not surprising; the regularity of any monomial ideal is bounded below by the maximal degree of its generators, and hence any nearly complete intersection containing a $k$-edge must have regularity at least $k$. More surprising perhaps is the fact that the skeleton of a nearly complete intersection may have arbitrarily large regularity, as demonstrated by the following class of examples generated in degree 2.
        
        \begin{theorem}
        Let $G_n = (V_n, E_n)$ where $V_n = \{c, a_1, b_1, \dots, a_n, b_n\}$ and $E = \{(c, a_i), (c, b_i), (a_i, b_i) \vert 1 \leq i \leq n\}$. Then $\text{reg}(R/I(G_n)) = n$.
        \end{theorem}
        
        For this proof, we use the notions of $\text{ind-match}$ and $\text{min-match}$ from \cite{SeyedFakhari2017}. 
        
        Let $G = (V, E)$ be a graph. We say that a \defi{matching} of $G$ is a set of edges $M \subseteq E$ so that no two edges in $M$ share a vertex. The \defi{matching number} of $G$, denoted $\text{match}(G)$, is the largest cardinality of a matching of $G$.
        
        Now let $\mathcal{H}$ be a collection of graphs. An \defi{$\mathcal{H}$-subgraph} of $G$ is a subgraph of $G$ whose connected components belong to $\mathcal{H}$. An \defi{induced $\mathcal{H}$-subgraph} of $G$ is an $\mathcal{H}$-subgraph which is an induced subgraph of $G$. An $\mathcal{H}$-subgraph $H = (V', E')$ of $G$ is called \defi{maximal} if the subgraph induced on $V \backslash V'$ contains no nonempty $\mathcal{H}$-subgraphs.
        
        The authors of \cite{SeyedFakhari2017} define $\text{ind-match}$ and $\text{min-match}$ as follows:
        \begin{align*}
            \text{ind-match}_{\mathcal{H}}(G) &:= \max{\left\{\text{match}(H) \, \vert\, H \text{ is an induced } \mathcal{H}\text{-subgraph of } G\right\}}; \\
             \text{min-match}_{\mathcal{H}}(G) &:= \min{\left\{\text{match}(H) \, \vert\, H \text{ is a maximal } \mathcal{H}\text{-subgraph of } G\right\}}.
        \end{align*}
        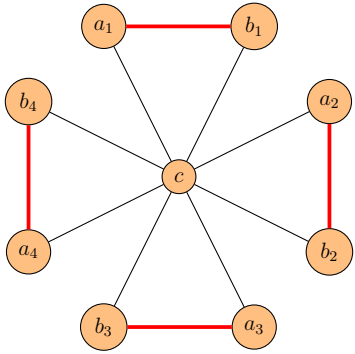
\begin{figure}
            \centering
            \begin{tikzpicture}
              		[scale=1,auto=left,every node/.style={circle,fill=orange!50, scale=0.7, draw=black},every loop/.style={}]
              		\node (c) at (0, 0) {$c$};
              		\node (a1) at (-1, 2) {$a_1$};
              		\node (b1) at (1, 2) {$b_1$};
              		\node (a2) at (2, 1) {$a_2$};
              		\node (b2) at (2, -1) {$b_2$};
              		\node (a3) at (1, -2) {$a_3$};
              		\node (b3) at (-1, -2) {$b_3$};
              		\node (a4) at (-2, -1) {$a_4$};
              		\node (b4) at (-2, 1) {$b_4$};
              	
              		\draw (c)  -- (a1);
              		\draw (c)  -- (b1);
              		\draw[line width=0.5mm, red] (a1) -- (b1);
                  	\draw (c)  -- (a2);
              		\draw (c)  -- (b2);
              		\draw[line width=0.5mm, red] (a2) -- (b2);
                  	\draw (c)  -- (a3);
              		\draw (c)  -- (b3);
              		\draw[line width=0.5mm, red] (a3) -- (b3);
                  	\draw (c)  -- (a4);
              		\draw (c)  -- (b4);
              		\draw[line width=0.5mm, red] (a4) -- (b4);
                \end{tikzpicture}
            \caption{$G_4$ with the corresponding matching set in red.}
            \label{fig:G-graphs}
        \end{figure}
        \begin{proof}
        Use the inequality 
        \[\text{ind-match}_{\{K_2, C_5\}}(G) \leq \text{reg}(R/I(G)) \leq \text{min-match}(G)_{\{K_2, C_5\}}
        \]
        given in \cite[Corollary 3.9]{SeyedFakhari2017}. Let $\mathcal{H} = \{K_2, C_5\}$. Since $G_n$ contains no subgraphs isomorphic to $C_5$, it follows that the only $\mathcal{H}$-subgraphs of $G_n$ are subgraphs consisting of disjoint copies of $K_2$. 
        The only maximal $\mathcal{H}$-subgraphs of $G_n$ are $H = (V, E)$ where $V$ and $E$ take one of the following 3 forms:
        \begin{enumerate}
            \item $V_0 = \{a_i, b_i \, | \, 1 \leq i \leq n\}$ and $E_0 = \{\{a_i, b_i\} \, | \, 1 \leq i \leq n\}$.
            \item $V_{b_i} = V_0 \setminus \{b_i\} \cup \{ c\}$ and $E_{b_i} = E_0 \backslash \{\{a_i, b_i\}\} \cup \{\{a_i, c\}\}$.
            \item $V_{a_i} = V_0 \setminus \{a_i\} \cup \{ c\}$ and $E_{a_i} = E_0 \backslash \{\{a_i, b_i\}\} \cup \{\{b_i,c\}\}$.
        \end{enumerate}
        In any of these cases, the matching number of the maximal $\mathcal{H}$-subgraph is $n$, so $\text{min-match}_{\{K_2, C_5\}}(G_n) = n$.
        
        Similarly, the subgraph induced on $\{a_1, b_1, \dots, a_n, b_n\}$ has maximal matching number among all induced $\mathcal{H}$-subgraphs. Thus, $\text{reg}(R/I(G_n)) = n$. See Figure~\ref{fig:G-graphs} for a graphical representation of $G_4$ and its corresponding matching set.
        \end{proof}
    
    \section{Appendix: Vertex-Weighted Graphs}\label{sec:VerWeightGraphs}
    
        Theorem~\ref{thm:main} characterizes NCI-hypergraphs as having an extremely simple hyperedge structure, a visual fact quantified by our discussion of the Betti tables of NCIs in the body of the paper. In particular, NCI-hypergraphs must be joinable, with hyperedges that resemble ``fat vertices." This notion naturally leads us to connect joinable hypergraphs with vertex-weighted graphs.
        
        \begin{definition}
            Given a joinable hypergraph $G$, the \defi{join} of $G$ at a hyperedge $h = (v_1,...,v_n)$ is the hypergraph $G' = \textrm{Join}_h(G)$ defined by replacing all vertices in $h$ with a single vertex $v$ of weight $n$. If all elements of $h$ are connected to a vertex $w$ in $G$, then $v$ is connected to $w$ in $G'$.
        \end{definition}
        
        This is identical to the definition of join in the main paper, except we also label the new vertex with a weight using the cardinality of the $k$-edge it came from. If we take the join of $G$ at all of its hyperedges, we are left with a unique vertex-weighted graph. In fact, join defines a bijection between joinable hypergraphs and vertex-weighted graphs, as can be seen by defining its inverse operation. Given a vertex-weighted graph $G$, the \defi{splay} of $G$ is the hypergraph $G'$ obtained by replacing each vertex $v\in G$ of weight $n>1$ with the $n$ vertices $v_1,...,v_n$. Each $v_i$ is connected to every element of $N(v)$, and the hyperedge $e = (v_1,...,v_n)$ is included in $G'$ as well. The definition of vertex inversion on a vertex-weighted graph is induced directly from the definition of vertex inversion on the associated joinable hypergraph via the join and splay operations.
        
        \begin{figure}
            \centering
            \begin{tikzpicture}
              		[scale=1,auto=left,every node/.style={circle,fill=orange!50, scale=0.7, draw=black},every loop/.style={},baseline={(n2.base)}]
              		\node (n1) at (0, 3) {};
              	    \node (n2) at (0, 2) {};
              		\node (n3) at (0, 1) {};
              		\node (n7) at (2, 2.5) {};
              		\node (n8) at (2, 1.5) {};
              		\node (n4) at (4, 3) {};
              		\node (n5) at (4, 2) {};
              		\node (n6) at (4, 1) {};
              	
              		\draw (n7) -- (n1);
              		\draw (n7) -- (n2);
              		\draw (n7) -- (n3);
              		\draw (n7) -- (n4);
              		\draw (n7) -- (n5);
              		\draw (n7) -- (n6);
              	    \draw (n8) -- (n1);
              		\draw (n8) -- (n2);
              		\draw (n8) -- (n3);
              		\draw (n8) -- (n4);
              		\draw (n8) -- (n5);
              		\draw (n8) -- (n6);
              		\draw (n7) -- (n8);
              		
              		\begin{pgfonlayer}{background}
                    \begin{scope}[transparency group, opacity=0.5]
                    \fill[edge,opacity=1,color=green, line width=25pt] (n1.center) -- (n2.center) -- (n3.center) -- (n1.center);
                    \fill[edge,opacity=1,color=green, line width=25pt] (n4.center) -- (n5.center) -- (n6.center) -- (n4.center);
                    \end{scope}
                    \end{pgfonlayer}
          	    \end{tikzpicture}
          	    $\xleftrightarrow{\hspace{2cm}}$
          	    \begin{tikzpicture}
              		[scale=1,auto=left,every node/.style={circle,fill=orange!50, scale=0.7, draw=black},every loop/.style={},baseline={(n1.base)}]
              		\node (n1) at (0, 0) {$3$};
              	    \node (n2) at (2, 1) {$1$};
              	    \node (n2p) at (2, -1) {$1$};
              		\node (n3) at (4, 0) {$3$};

              		\draw (n1) -- (n2);
              		\draw (n2) -- (n3);
              		\draw (n1) -- (n2p);
              		\draw (n2p) -- (n3);
              		\draw (n2) -- (n2p);
          	    \end{tikzpicture}
            \caption{A joinable hypergraph with its corresponding vertex-weighted graph representation.}
            \label{fig:splay-join-example}
        \end{figure}
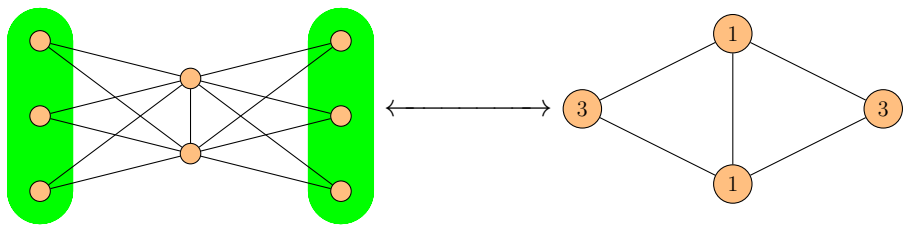
        
        Consistent with \cite{StoneMiller}, we define a \defi{CI-weighted graph} to consist of isolated edges between vertices of weight 1 and isolated vertices of weight greater than 1. It follows immediately that the splay of a weighted graph $G$ is a CI-hypergraph if and only if $G$ is a CI-weighted graph. This connection is enough for us to generalize the Miller-Stone characterization \cite{StoneMiller} of NCI-graphs in degree 2 to weighted graphs.
        
        \begin{lemma}
            A weighted graph $G$ on at least 3 vertices is an NCI-weighted graph if and only if it satisfies the following properties:
            \begin{itemize}
                \item $G$ does not contain the path on 4 vertices as a subgraph, where one end has weight greater than 1;
                \item $G$ satisfies the Miller-Stone condition when viewed as an unweighted graph.
            \end{itemize}
        \end{lemma}
        
        NCI-weighted graphs and NCI-hypergraphs exist in close correspondence: the splay of an NCI-weighted graph is an NCI-hypergraph, and the join of an NCI-hypergraph is an NCI-weighted graph. Representing nearly complete intersections as weighted graphs drastically reduces the number of edges involved, making it easier to see underlying structures and properties.
        
        The reader who likes to tinker with these ideas may enjoy our graph builder tool:\\ \url{https://crgibbons.github.io/files/WeightedGraphBuilder}\textbf{}
    
    \bibliography{Refs/biblio}
    \bibliographystyle{plain}

\end{document}